\documentclass[a4paper,10pt]{amsart}
\usepackage{amsthm,amsmath,amsfonts,amscd,amssymb} 
\usepackage[all]{xy}  
\usepackage{hyperref} 
\usepackage{pdfsync}   
\usepackage{amsthm}   
\usepackage{color}
\theoremstyle{plain}
\newtheorem{theorem}{Theorem}[section]
\newtheorem{lemma}[theorem]{Lemma}

\newtheorem{conjecture}[theorem]{Conjecture}

\theoremstyle{definition}

\newtheorem{remark}[theorem]{Remark}

\DeclareMathOperator{\an}{an}

\newcommand{\IZ}{\mathbb{Z}}
\newcommand{\vol}{\operatorname{vol}}
\newcommand{\diam}{\operatorname{diam}}
\newcommand{\degree}{\operatorname{deg}}
\newcommand{\coker}{\operatorname{coker}}

{\catcode`@=11\global\let\c@equation=\c@theorem}

\newcounter{commentcounter}
\newcommand{\commenth}[1] 
{\stepcounter{commentcounter}
   \textbf{Comment \arabic{commentcounter} (by Herbert)}: 
{\textcolor{blue}{#1}} }

\newcommand{\commentw}[1] 
{\stepcounter{commentcounter}
   \textbf{Comment \arabic{commentcounter} (by Wolfgang)}: 
{\textcolor{red}{#1}} }

\newcommand{\version}[1]                    
{\begin{center} last edited on #1\\
last compiled on \today\\
name of texfile: \jobname
\end{center}
}

\title[On the spectral density function of the Laplacian of a graph]
{On the spectral density function of the Laplacian of a graph}
\author{Herbert Koch}
\email{koch@math.uni-bonn.de}
          \urladdr{http://www.math.uni-bonn.de/~koch/}
\author{Wolfgang L\"uck}
\email{wolfgang.lueck@him.uni-bonn.de}
          \urladdr{http://www.him.uni-bonn.de/lueck}
\address{Mathematisches Institut der Universit\"at Bonn\\
                Endenicher Allee 60\\
                53115 Bonn, Germany}
 \date{May 2012}     


\begin{document}

\begin{abstract}
  Let $X$ be a finite graph. Let $E$ be the number of its edges and $d$ be its
  degree.  Denote by $F_1(X)$ its first spectral density function which counts
  the number of eigenvalues $\le \lambda^2$ of the associated Laplace
  operator. We prove the  estimate $F_1(X)(\lambda) -F_1(X)(0) \le 2
  \cdot E \cdot d \cdot \lambda$ for $0 \le \lambda < 1$. We explain how this
  gives evidence for conjectures about approximating Fuglede-Kadison
  determinants and $L^2$-torsion.
\end{abstract}

\maketitle


\section{Introduction}

Our main result is

\begin{theorem}[Estimate on the first spectral density function]
 \label{the:main_theorem}
   Let $X$ be a finite graph. Let $|E(X)|$ be the number of its edges and $\degree(X)$ be its degree.
   Denote by $F_1(X)$ its first spectral density function.

   Then we obtain for $0 \le \lambda < 1$
   \[
   F_1(X)(\lambda) -F_1(X)(0) 
   \le  2 \cdot |E(X)| \cdot \degree(X)  \cdot \lambda.
    \]
 \end{theorem}
  
The notion of degree and spectral density function will be recalled in 
Section~\ref{sec:Estimate_on_the_first_spectral_density_function_of_a_finite_graph},
where also the proof of Theorem~\ref{the:main_theorem} is presented.

Our main motivation is the following conjecture.  For information about torsion
and $L^2$-torsion we refer to~\cite[Chapter~3]{Lueck(2002)}. We will explain in
Section~\ref{sec:higher_dimensions} why Theorem~\ref{the:main_theorem} gives
some evidence for it.

\begin{conjecture}[Approximation Conjecture for analytic $L^2$-torsion]
\label{con:Approximation_conjecture_for_analytic_L2-torsion}\ \\
Let $G$ be a group together with a sequence of in $G$ normal subgroups with finite index $[G:G_i]$
\[
G = G_0 \supseteq G_1 \supseteq G_2 \supseteq G_3 \supseteq \cdots
\]
such that $\bigcap_{i \ge0} G_i = \{1\}$. 
Let $M$ be a  closed Riemannian manifold and let $\overline{M} \to M$ be a $G$-covering.
Equip $\overline{M}$ and $M_i := \overline{M}/G_i$ with the Riemannian metrics 
for which the projections $\overline{M} \to M$ and $M_i \to M$ are local isometries.
Denote by $\rho^{(2)}_{\an}\big(\overline{M};{\mathcal N}(G)\bigr)$ the analytic $L^2$-torsion 
with respect to the cocompact proper free isometric $G$-action on $\overline{M}$
and denote  by $\rho_{\an}(M[i])$ the (analytic) Ray-Singer torsion of $M_i$.

Then
\[\rho^{(2)}_{\an}\big(\overline{M};{\mathcal N}(G)\bigr) 
= \lim_{i \to \infty} \;\frac{\rho^{(2)}_{\an}(M[i])}{[G:G_i]}.
\]
\end{conjecture}

\begin{remark}[Conjecture by Bergeron-Venkatesh]
  \label{rem:Bergeron-Venkatesh}
  Conjecture~\ref{con:Approximation_conjecture_for_analytic_L2-torsion} is
  related to a conjecture by
  Bergeron-Venkatesh~\cite[Conjecture~1.3]{Bergeron-Venkatesh(2010)} which
  expresses $L^2$-torsion in terms of the growth of the torsion of the singular
  homology of $M_i$ in dimension $n$ if $M$ is a locally symmetric space of
  dimension $\dim(M) = 2n+1$. Their conjecture makes also sense if one requires
  only that $M$ is an aspherical closed $(2n+1)$-dimensional manifold.
\end{remark}

The authors want to thank Jens Vygen for fruitful
discussions and hints.  The paper was financially support by the Leibniz-Award
of the second author granted by the DFG.


\section{Estimate on the first spectral density function of a finite graph}
\label{sec:Estimate_on_the_first_spectral_density_function_of_a_finite_graph}


\subsection{The spectral density function}

Let $X$ be a finite directed graph with set of vertices $V= V(X)$ and of edges
$E = E(X)$, where each edges $e$ has an initial vertex $v_0(e)$ and a terminal vertex  $v_1(e)$. 
 Given a vertex $v \in V$, we denote by $d_v$ its
\emph{degree}, i.e., the sum of the number of edges which have
precisely one endpoint equal to $v$ and two times the number of edges
whose two endpoints agree with $v$.  The \emph{degree} $\degree(X)$ of
$X$ is the maximum of the set $\{d_v \mid v \in V\}$.  Define the
\emph{volume} of $X$ to be
\begin{eqnarray}
  \vol(X) & = & \sum_{v \in V} d_v.
  \label{volume_of_a_finite_graph}
\end{eqnarray}

The elementary so called \emph{handshaking lemma} says
\begin{eqnarray}
  \vol(X) & = & 2 \cdot |E(X)|.
  \label{vol(X)_is_2|E|}
\end{eqnarray}

We equip the graph with the \emph{path metric}, i.e., the distance of
two points is the infimum over the length of all piecewise linear
paths joining these points, where the length of a piecewise linear path
is defined in the obvious way such that every edge has length one. The
\emph{diameter} $\diam(X)$ is the maximum of the distances of any two
vertices. Obviously we have
\begin{eqnarray}
\diam(X) & \le & |E(X)|.
\label{diam_le_|E(T)|}
\end{eqnarray}

We define $C_1 (X)=l^2(E(X))$ and $C_0 (X)=l^2(V(X))$ to be the vector space of 
sequences of real numbers indexed by the elements in $E(X)$ and
$V(X)$, equipped with the standard Euclidean inner product. We obtain a
$1$-dimensional chain complex whose first differential is 
\[ 
c_1 \colon C_1(X) \to C_0(X), \quad  (a_e)_{e\in E(X)} \mapsto
\left( \sum_{e: v_1(e)=v} a_e - \sum_{e : v_0(e) = v} a_e
\right)_{v\in V(X)}.
\]
We denote the \emph{adjoint} of $c_1$ by $c_1^* : C_0 (X) \to C_1(X)$. 
The \emph{zeroth Laplace operator of $X$} 
\[
\Delta_0 \colon C_0(X) \to C_0(X)
\]
is given by $c_1 \circ c_1^*$.  

   The \emph{spectral density function $F(f)$} of a map $f \colon V \to W$ 
   of finite-dimensional Hilbert spaces is the function
   \[
   F(f) \colon [0,\infty) \to [0,\infty)
   \]
   which sends $\lambda \in [0,\infty)$ to the supremum of the
   dimensions of all subvector spaces $V_0 \subseteq V$ for which $|f(v)| \le \lambda \cdot |v|$ holds for all 
   $v \in V_0$.
  The \emph{first spectral density function} of $X$ is  
   \[
   F_1(X) := F\bigl(c_1 \colon C_1(X) \to C_0(X)\bigr).
   \]
We will be interested in $F_1(X)(\lambda)-F_1(X)(0)$.

The dimension of the kernel of $c_1$ is called the \emph{first Betti number}
$b_1(X)$,
and  the dimension of the kernel  of $c_1^*$ is called the
\emph{zeroth Betti number} $b_0(X)$, which is also the number of
components of $X$.


\subsection{Singular values}

The \emph{singular value decomposition} states that, given a $n \times m$ matrix $A$, there exist
orthogonal matrices $U$ und $V$ and a sequence of nonnegative numbers
$\sigma_i$ called singular values so that
\begin{eqnarray}
A  &= &     U^T  \left( \begin{matrix} \sigma_1 & 0 & \dots & 0 \\
                                         0    & \sigma_2 & 0 & \dots  \\
                                         \vdots & \ddots &   & \vdots 
\end{matrix}\right) V. 
\label{singular_value_decomposition}
\end{eqnarray}
The number $\sigma_i$ are uniquely determined up to their order.
They are called \emph{singular values}. Equivalently, if $f\colon  V \to W$ is
a linear map between finite dimensional Euclidean vector spaces, then
there exist an orthonormal basis in $V$ and an orthonormal basis in $W$ so that 
the matrix representing $f$ in that basis is the  diagonal 
matrix appearing in~\eqref{singular_value_decomposition}.

The nonzero eigenvalues 
of $A^T A$ and $AA^T$ coincide. They are the
squares of the nonzero singular values. The matrices $A$ and $A^T$
have the same nonzero singular values. 
The number of non-zero singular values of
$f$ less or equal to  $\lambda$ is $F(f)(\lambda) - F(f)(0)$.

We collect some of the  elementary statements discussed 
above. 

\begin{lemma} \label{lem_basics_about_spectral_density} We get for
  $\lambda \ge 0$:
   \begin{enumerate}
   \item \label{lem_basics_about_spectral_density:dual} 
    $F_1(X)(\lambda)
     - F_1(X)(0) = F(c_1^*)(\lambda) - F(c_1^*)(0);$
    \item \label{lem_basics_about_spectral_density:Laplace}
    $F\bigl(c_1^*\bigr)(\lambda) = F(\Delta_0)(\lambda^2)$
    is the sum of all eigenvalues of $\Delta_0$
   counted with multiplicity which are less or equal to $\lambda^2$;
    \item \label{lem_basics_about_spectral_density:Betti_numbers} 
    We have
    \[b_0(X) := \dim(\ker(c_1^*)) = F(c_1^*)(0)  = F(\Delta_0)(0) = \dim(\coker(c_1)),\]
     and
     \[b_1(X) := \dim(\ker(c_1)) = F_1(X)(0).\]
         \end{enumerate}
  \end{lemma}

See~\cite[Lemma~2.4 on page~74, Lemma~2.11~(11) on page~77, Example~2.5
on page~75, Lemma~1.18 on page~24 ]{Lueck(2002)} for analogous statements to
Lemma~\ref{lem_basics_about_spectral_density}.

Observe that the Laplace operator and $F_1(X)$ do not depend on the orientation of the edges.


\subsection{One-dimensional CW complexes}
A finite one-dimensional CW complex is the same as a finite graph. One-cells
correspond to edges, and zero-cells to vertices.  A finite one-dimensional CW
complex with an orientation for each cell is the same as a directed graph.  The
cellular chain complex with real coefficients agrees with the chain complex
$C_*(X)$ of the graph introduced above.  The Betti numbers above are  the
Betti numbers of the CW complex defined as the dimensions of its homology
groups.


\subsection{Formulation of the main theorem}
 We want to show

    \begin{theorem}[Estimate on the first spectral density function]
    \label{the:Estimate_on_the_first_spectral_density_function}
    Let $X$ be a finite connected graph.  

    \begin{enumerate} 
    \item \label{the:Estimate_on_the_first_spectral_density_function:degree_ge_2} 
    Suppose that 
    $\deg(X) \ge 2$.  Then we get 
    \begin{multline*}
    F_1(X)(\lambda) -F_1(X)(0) 
    \\
     \le
    \begin{cases}
    0  & \text{if}\; \lambda < \frac{1}{\sqrt{2} \cdot |E(X)|};
    \\
    2 \cdot |E(X)| \cdot \degree(X)  \cdot \lambda
& \text{if}\;  \frac{1}{2 \cdot (|E(X)|-1) \cdot \deg(X)} \le \lambda < 1;
  \end{cases}
\end{multline*}
\item \label{the:Estimate_on_the_first_spectral_density_function:degree_le_1} 
   If $\degree(X) \le 1$, then $X$ is a point or the unit interval and 
   $F_1(X)(\lambda) - F_1(X)(0) = 0$ for $0 \le \lambda \le 1$;

   \item \label{the:Estimate_on_the_first_spectral_density_function:desired_estimate}
   We obtain for $0 \le \lambda < 1$
   \[
   F_1(X)(\lambda) -F_1(X)(0) 
   \le  2 \cdot |E(X)| \cdot \degree(X)  \cdot \lambda.
    \]
 \end{enumerate}
\end{theorem}
  
For connected graphs $F_1(X)(0) = 1$ and the statement could be simplified. 
We will take that into account in the sequel, but we still prefer to state 
the main result in a formulation motivated by the corresponding 
vector valued problem. This context will be discussed in the next section. 


\subsection{Proof of the main Theorem}

For the proof of
Theorem~\ref{the:Estimate_on_the_first_spectral_density_function} we
need the following three results. The first one is taken
from~\cite[Lemma~1.9]{Chung(1997)}.

\begin{theorem}[Estimate on the first non-trivial eigenvalue]
\label{the:Estimate_on_the_first_non-trivial_eigenvalue}
Let $\lambda_1$ be the smallest non-zero eigenvalue of the zero-th Laplace operator $\Delta_0$ of $X$.
Then
\[
\lambda_1 \ge \frac1{\diam(X) \cdot  \vol(X)}.
\]
\end{theorem}

\begin{lemma}\label{lem_spectral_density_and_deleting_vertices}
  Let $X$ be a finite graph and $Y \subseteq X$ be a subgraph obtained from $X$ by
  deleting some edges. Then we get for $0 \le \lambda$
\[
F\bigl(c_1(X)^*\bigr)(\lambda) \le F\bigl(c_1(Y)^*\bigr)(\lambda).
\]
\end{lemma}
\begin{proof}
  Let $i \colon C_1(Y) \to C_1(X)$ be the inclusion. It induces an
  isometric embedding $i_* \colon C_1(Y) \to
  C_1(X)$. Since $c_1(Y) = c_1(X) \circ i_*$, the claim follows directly from the definitions.
\end{proof}

\begin{lemma} \label{lem:splitting_a_tree_into_a_forest} Let $T$ be a
  finite tree. Let $P$ be a real number with 
  $P \le (|E(T)|-1)\cdot  \degree(T)$.  Then we can remove an appropriate collection of edges
  such that the resulting forrest  is a disjoint union 
  $\coprod_{i = 1} ^k T_i$ of trees $T_1$, $T_2$, \ldots, $T_k$ with the property
  that $1 \le k \le \frac{]E(T)| \cdot \degree(T)}{P} +1$ and $E(T_i) \le
  P$ for $i = 1,2, \ldots , k$ holds.
\end{lemma}
\begin{proof}
  Next we describe the following construction on $T$.  Choose a leaf
  $v$, i.e., a vertex $v$ with $\degree(v) = 1$. Such a leaf always
  exists in a finite tree. Since $|E(T)| -1 \ge \frac{P}{\degree (T)}$,
  there exists an edge $e$ with the property (P) that after removing
  $e$ the tree splits into to trees $T'$ and $T''$ such that $T'$ contains at least $\frac{P}{\degree(T)}$
  edges and $T''$  contains $v$. 
  Namely, take for instance for $e$ the edge which has $v$ as vertex. Now choose an
  edge $e$ which has  property (P) and which has among all edges
  with  property (P) maximal distance from $v$.

  We claim that $T'$ contains at most $P$ edges.  Let $v'$ be the
  vertex in $T'$ which belongs to $e$. Let $e_1, e_2, \ldots, e_l$ be
  the edges in $T'$ which have $v'$ as vertex. Since in $T$ the
  degree of $v'$ is bounded by $\degree(T)$, we have $l \le
  \degree(T) -1$. If we remove $e_i$ from $T$, we obtain a disjoint
  union of trees $T_i' \amalg T_i''$ such that $T_i''$ contains
  $v$. The distance in $T$ of $e_i$ from $v$ is larger than the distance in $T$ of
  $e$ from $v$. Hence the tree $T_i'$ contains at most
  $\frac{P}{\degree (T)} -1$ edges.  Every edge of $T'$ is an edge of
  $T_i'$ for some $i \in \{1,2, \ldots, l\}$ or belongs to 
  $\{e_1,  e_2, \ldots,  e_l\}$. Hence
  \begin{multline*}
    |E(T')| \le l + \sum_{i =1}^l |E(T_i'| 
    \le l + \sum_{i=1}^l \left(\frac{P}{\degree (T)} -1\right) 
    \\
    = l + l \cdot
    \left(\frac{P}{\degree (T)} -1\right)
    = l \cdot \frac{P}{\degree (T)} \le \frac{P \cdot (\degree(T)
      -1)}{\degree (T)} < P.
  \end{multline*}
  Now put $T_1 := T'$. Recall that $T_1$ has at least $
  \frac{P}{\degree (T)}$ and at most $P$ edges.  If $P > (|E(T'')|-1)
  \cdot \degree(T)$, we put $T_2 = T''$ and $k = 2$.  Obviously $T_2$
  has at most $P$ edges.  If $P \le  (|E(T'')|-1)  \cdot \degree(T)$, 
  we apply the procedure to $T''$
  and the resulting tree is a disjoint union of a tree $T_2$ and a
  tree $T'''$ such that $T_2$ has at least $ \frac{P}{\degree (T)}$
  and at most $P$ edges.  Either $T'''$ becomes $T_3$ and the
  procedure stops, or we apply the construction to $T'''$.  This
  procedure has to stop after finitely many steps, since $T$ contains only finitely
  many edges. Thus we have
  removed edges from $T$ such that the resulting tree is a 
  disjoint union $\coprod_{i = 1} ^k T_i$ of trees $T_1$, $T_2$,
  \ldots, $T_k$ with the property that $|E(T_i)| \le P$ for 
  $i = 1,2,   \ldots , k$ and $\frac{P}{\degree (T)} \le |E(T_i)|$ holds for 
  $i =  1,2 , \ldots , (k-1)$.  The latter inequality implies 
  $1 \le k \le \frac{]E(T)| \cdot \degree(T)}{P} +1$ since the sum
  of the number of edges in $T_i$ for $i = 1,2, \ldots, (k-1)$ is bounded by $|E(T)|$.
\end{proof}

Now we are ready to prove
Theorem~\ref{the:Estimate_on_the_first_spectral_density_function}.
\begin{proof}[Proof of
  Theorem~\ref{the:Estimate_on_the_first_spectral_density_function}]\
  \\~\eqref{the:Estimate_on_the_first_spectral_density_function:degree_le_1}
  This follows from a direct elementary calculation.
  \\[1mm]~\eqref{the:Estimate_on_the_first_spectral_density_function:degree_ge_2}
  Let $T \subseteq X$ be a maximal tree. Then   
$\degree(T) \le \degree(X)$, $\diam(T) \le \diam(X)$, $|E(T)| \le  |E(X)|$,  $|V(T)| = |V(X)|$,
and $b_0(X) = b_0(T) = 1$. We have
  \begin{eqnarray*}
   & \frac{1}{\sqrt{2} \cdot |E(X)|}  \le  \frac{1}{\sqrt{2} \cdot |E(T)|};&
   \\
   & \frac{1}{2 \cdot (|E(X)|-1) \cdot \deg(X)}  \le   \frac{1}{\sqrt{2} \cdot |E(T)|};&
   \\
    &\frac{1}{2 \cdot (|E(T)|-1) \cdot \deg(T)} \le \frac{1}{\sqrt{2} \cdot |E(T)|}\quad \text{if}\, |E(T) \ge 2;&
   \\
   & 2 \cdot |E(T)| \cdot \degree(T)  \cdot \lambda \le 2 \cdot |E(X)| \cdot \degree(X)  \cdot \lambda.&
 \end{eqnarray*}
  We conclude from Lemma~\ref{lem_basics_about_spectral_density} and
  Lemma~\ref{lem_spectral_density_and_deleting_vertices} 
  \begin{eqnarray*}
  F_1(X)(\lambda) - F_1(X)(0)
  & = & 
  F\bigl(c_1(X)^*\bigr)(\lambda) -  F\bigl((c_1(X)^*\bigr)(0)
  \\
  & = & 
  F\bigl(c_1(X)^*\bigr(\lambda) -  b_0(X)
  \\
  &\le &
    F\bigl(c_1(T)^*\bigr)(\lambda) -  b_0(X)
  \\
  &= &
    F\bigl(c_1(T)^*\bigr)(\lambda) -  b_0(T)
  \\
  & =  &
    F\bigl(c_1(T)^*\bigr(\lambda) -  F(c_1(T)^*(0)
   \\
   & = & 
   F_1(T)(\lambda) - F_1(T)(0).
 \end{eqnarray*}
 The inequalities above and assertion~\eqref{the:Estimate_on_the_first_spectral_density_function:degree_le_1} 
show that it suffices to prove the claim in the special case $X = T$
and $\degree(T) \ge 2$.  Notice that $\degree(T) \ge 2$ implies $|E(T)| \ge 2$.

  Suppose that $0 \le \lambda <\frac{1}{\sqrt{2} \cdot E(T)}$. Let $\lambda_1(T)$
  be the smallest non-zero eigenvalue of $\Delta_0$ on $T$. Since 
  $\vol(T) = 2   \cdot |E(T)|$ by~\eqref{vol(X)_is_2|E|} and $\diam(T) \le |E(T)|$
  by~\eqref{diam_le_|E(T)|} , we conclude from
  Theorem~\ref{the:Estimate_on_the_first_non-trivial_eigenvalue}.
  \[
  0 \le \lambda^2 < \frac{1}{2 \cdot |E(T)|^2} \le \frac{1}{\diam(T) \cdot
    \vol(T)} \le \lambda_1(T),
  \]
  and hence
  \[F(\Delta_0)(\lambda^2) - F(\Delta_0)(0) = 0.
\] 
   We conclude from Lemma~\ref{lem_basics_about_spectral_density} 
  \[F_1(T)(\lambda) - F_1(T)(0) = 0.\] 
  Now suppose 
  $\frac{1}{2 \cdot (E(T)-1) \cdot \deg(T)} \le \lambda< 1$. Put
  \[P := \frac{1}{2 \cdot \lambda}.
  \]
  Then $P \le (|E(T)|-1)\cdot \degree(T)$. Hence we can apply
  Lemma~\ref{lem:splitting_a_tree_into_a_forest}.  So we can remove edges from
  $T$ such that the resulting tree is a disjoint union $\coprod_{i = 1} ^k T_i$
  of trees $T_1$, $T_2$, \ldots, $T_k$ with the property that 
  $1 \le k \le  \frac{|E(T)| \cdot \degree(T)}{P} +1$ and $|E(T_i)| \le P$ for $i = 1,2, \ldots
  ,k$ holds. Let $\lambda_1(T_i)$ be the smallest non-zero eigenvalue of
  $\Delta_0(T_i)$ on $T_i$ for $i = 1,2, \ldots, k$. Since 
   $\vol(T_i) = 2 \cdot  |E(T_i)|$ by~\eqref{vol(X)_is_2|E|} and $\diam(T_i) \le |E(T_i)|$
  by~\eqref{diam_le_|E(T)|}, we conclude from
  Theorem~\ref{the:Estimate_on_the_first_non-trivial_eigenvalue} for $i = 1,2,
  \ldots, k$
  \[
  \lambda_1(T_i) \ge \frac{1}{\diam(T_i) \cdot \vol(T_i)} 
  \ge \frac{1}{2 \cdot  |E(T_i)|^2} \ge \frac{1}{2 \cdot P^2} 
   > \frac{1}{4 \cdot P^2} = \lambda^2.
  \]
 Hence $F\bigl(\Delta_0(T_i)\bigr)(\lambda^2) - F\bigl(\Delta_0(T_i)\bigr)(0)  = 0$.    
  Lemma~\ref{lem_basics_about_spectral_density}   implies 
  \[   
   F\bigl(c_1(T_i)^*\bigr)(\lambda) - F\bigl(c_1(T_i)^*\bigr)(0)
  =   F\bigl(\Delta_0(T_i)\bigr)(\lambda^2) - F\bigl(\Delta_0(T_i)\bigr)(0) = 0.
  \]
  This together with Lemma~\ref{lem_basics_about_spectral_density} and
  Lemma~\ref{lem_spectral_density_and_deleting_vertices} implies
  \begin{eqnarray*}
    F_1(T)(\lambda)   
    & = & 
    F_1(T)(\lambda)   - b_1(T)
    \\
     & = & 
    F_1(T)(\lambda)   - F_1(T)(0)
    \\
    & = & 
    F\bigl(c_1(T)^*\bigr)(\lambda)  -  F\bigl(c_1(T)^*\bigr)(0)
    \\
    & \le &
     F\bigl(c_1(\amalg_{i = 1} ^k T_i)^*\bigr)(\lambda)  -  F\bigl(c_1(T)^*\bigr)(0)
     \\
     & = & 
     \left(\sum_{i=1}^k F\bigl(c_1(T_i)^*\bigr)(\lambda) \right) -   F\bigl(c_1(T)^*\bigr)(0) 
     \\
      & = & 
     \left(\sum_{i=1}^k \left(F\bigl(c_1(T_i)^*\bigr)(\lambda) - F\bigl(c_1(T_i)^*\bigr)(0)\right)\right)
     \\ 
     & & \quad \quad +  \left(\sum_{i=1}^k F\bigl(c_1(T_i)^*\bigr)(0)\right) -   F\bigl(c_1(T)^*\bigr)(0) 
     \\ 
     & = & 
     \left(\sum_{i=1}^k F\bigl(c_1(T_i)^*\bigr)(0)\right)  -   F\bigl(c_1(T)^*\bigr)(0) 
      \\
      & = & 
      \left(\sum_{i=1}^k b_0(T_i) \right)-   b_0(T)
      \\
      & = &
      k -1
      \\
      & \le & 
      \frac{|E(T)| \cdot \degree(T)}{P}
       \\
       & = & 
       2 \cdot |E(T)| \cdot \degree(T) \cdot \lambda.
  \end{eqnarray*}
  This finishes the proof of assertion~\eqref{the:Estimate_on_the_first_spectral_density_function:degree_ge_2}.
  \\[1mm]~\eqref{the:Estimate_on_the_first_spectral_density_function:desired_estimate} 
  This follows from assertions~\eqref{the:Estimate_on_the_first_spectral_density_function:degree_ge_2}
  and~\eqref{the:Estimate_on_the_first_spectral_density_function:degree_le_1}
  since $\frac{1}{2 \cdot (|E(X)|-1) \cdot \deg(X)}  < \frac{1}{\sqrt{2} \cdot |E(X)|}$ holds for $\degree(X) \ge 2$.
\end{proof}

\begin{remark}[General strategy] \label{rem:general_strategy} Roughly
  speaking, our method of proof deduces from the estimate for the
  lowest non-zero eigenvalue (see
  Lemma~\ref{the:Estimate_on_the_first_non-trivial_eigenvalue}) an
  estimate about the spectral density function for small $\lambda$,
  i.e., an estimate for the distribution of the small eigenvalues, see
  Theorem~\ref{the:main_theorem}.  This strategy has already been carried
  out in the paper by Grigor'yan-Yau~\cite{Grigorian-Yau(2003)}, where
  the distribution of the small eigenvalues of an elliptic operators on a manifold is  studied,
  provided that one has an estimate on the first non-zero eigenvalue. 

  Probably one could obtain Theorem~\ref{the:main_theorem} by modifying the arguments 
  of~\cite{Grigorian-Yau(2003)} for graphs. Our proof, however, is independent, short, elementary (due to
  the more elementary situation), and it yields explicit constants which will be important for the applications in
  Section~\ref{sec:Approximating_the_Kadison-Fuglede_determinant_for_finite_graphs}.
 \end{remark}


\section{Approximating the Kadison-Fuglede determinant for finite graphs}
\label{sec:Approximating_the_Kadison-Fuglede_determinant_for_finite_graphs}

Let $G$ be a group together with a sequence of in $G$ normal subgroups with finite index $[G:G_i]$
\[
G = G_0 \supseteq G_1 \supseteq G_2 \supseteq G_3 \supseteq \cdots
\]
such that $\bigcap_{i \ge0} G_i = \{1\}$. Let $p \colon \overline{X} \to X$ be a
$G$-covering of the finite graph $X$ such that $\overline{X}$ is connected. Put
$X_i := \overline{X}/G_i$.  Then the projection $X_i \to X$ is a
$[G:G_i]$-sheeted regular covering of connected finite graphs.

Consider a linear map $f \colon V \to W$ of finite-dimensional Hilbert spaces.
Let $\det^{(2)}(f)$ be its \emph{Fuglede-Kadison determinant} with respect to
the trivial group in the sense of~\cite[Definition~3.11~on
page~127]{Lueck(2002)}.  Notice that we do not require $f$ to be injective or
surjective.  If $\lambda_1$, $\lambda_2$, $\ldots$, $\lambda_r$ are the non-zero
singular values of $f$, then
\begin{eqnarray*}
{\det}^{(2)}(f)
& = &
\prod_{i=1}^r \lambda_i.
\end{eqnarray*}
With our convention the Fuglede-Kadison determinant of the zero map 
$0 \colon V \to W$ is always $1$. If $f$ is injective, then $\det^{(2)}(f)$ is the
square root of the classical determinant of the positive automorphism
$f^*f\colon V \to V$. Notice that ${\det}^{(2)}(f)$ is a real number greater
than zero and hence $\ln({\det}^{(2)}(f))$ is a well-defined real number. 

Consider the differential $c_1(\overline{X}) \colon C_1(\overline{X}) \to C_0(\overline{X})$
in the cellular $\IZ G$-chain complex of the free cocompact $G$-$CW$-complex $\overline{X}$.
Its $\IZ G$- chain modules are finitely generated free.  Let 
$c_1^{(2)}(\overline{X}) \colon C_1^{(2)}(\overline{X}) \to C_0^{(2)}(\overline{X})$ be 
the associated morphism of finitely generated Hilbert
${\mathcal N}(G)$-modules, where ${\mathcal N}(G)$ is the group von Neumann
algebra. Let $\det^{(2)}\bigl(c_1^{(2)}(X_i);{\mathcal N}(G)\bigr)$ be the
associated Fuglede-Kadison determinant in the sense of~\cite[Definition~3.11~on
page~127]{Lueck(2002)}. Then

\begin{theorem}[Determinant approximation for graphs]
\label{the:Determinant_approximation_for_graphs}
We have under the conditions above
\[
\ln\left({\det}^{(2)}\bigl(c_1^{(2)}(\overline{X});{\mathcal N}(G)\bigr)\right) 
= \lim_{i \to \infty} \frac{\ln\bigl({\det}^{(2)}(c_1^{(2)}(X_i))\bigr)}{[G:G_i]}.
\]
\end{theorem}
\begin{proof}
  Put $C:= 2 \cdot |E(X)| \cdot \degree(X)$. Since $X_i \to X$ is a regular
  $[G:G_i]$-sheeted covering, we get
\begin{eqnarray*}
\degree(X_i) & = & \degree(X);
\\
\frac{|E(X_i)|}{[G:G_i]} & = & |E(X)|.
\end{eqnarray*}
Theorem~\ref{the:main_theorem} implies for $\lambda \in [0,1)$
\begin{eqnarray}
\frac{F_1(X_i)(\lambda) - F_1(X_i)(0)}{[G:G_i]} & \le & C \cdot \lambda.
\label{uniform_estimate}
\end{eqnarray}
Notice that $C$ is independent of $i$. This implies
\begin{eqnarray}
\int_0^1 \sup\left\{\left.\frac{F_1(X_i)(\lambda) 
- F_1(X_i)(0)}{[G:G_i] \cdot \lambda} \;\right|\; i = 1,2,\ldots \right\} d \lambda 
& < & \infty.
\label{Majorizing}
\end{eqnarray}
Now Theorem~\ref{the:Determinant_approximation_for_graphs} follows from a
general strategy which will be described in details  in~\cite{Lueck(2012_l2approx)} and 
is a consequence of the material in~\cite[Subsection~13.2.1]{Lueck(2002)}. 
The basic idea is, roughly speaking, that
\begin{eqnarray*}
\frac{\ln\bigl({\det}^{(2)}(c_1^{(2)}(X_i))\bigr)}{[G:G_i]} 
& = & 
- \int_0^K \frac{F_1(X_i)(\lambda) - F_1(X_i)(0)}{[G:G_i] \cdot \lambda} d \lambda 
\\
& & \quad \quad + \ln(K) \cdot \bigl(F_1(X_i)(K) - F_1(X_i)(0)\bigr);
\\
\ln\left({\det}^{(2)}\bigl(c_1^{(2)}(\overline{X});{\mathcal N}(G)\bigr)\right) 
& = & 
- \int_0^K \frac{F_1(\overline{X})(\lambda) - F_1(\overline{X_i})(0)}{\lambda} d \lambda 
\\
& & \quad \quad + \ln(K) \cdot \bigl(F_1(\overline{X})(K) - F_1(\overline{X})(0)\bigr);
\\
F_1(\overline{X})(K) - F_1(\overline{X})(0) & = & 
\lim_{i \to \infty} \bigl(F_1(X_i)(K) - F_1(X_i)(0)\bigr),
\end{eqnarray*}
holds for the $L^2$-spectral density function $F_1(\overline{X})(\lambda)$ of
the cocompact free $G$-$CW$-complex $\overline{X}$ and an large enough real number
$K$, and one has almost everywhere
\[
\frac{F_1(\overline{X})(\lambda) - F_1(\overline{X})(0)}{\lambda}
 =  
\lim_{i \to \infty}  
\frac{F_1(X_i)(\lambda) - F_1(X_i)(0)}{[G:G_i] \cdot \lambda},
\]
so that an application of the Lebesgue's Dominated Convergence Theorem finishes the proof,
provided that~\eqref{Majorizing} holds.
\end{proof}

\begin{remark}[First Novikov-Shubin invariant]
  \label{rem:Novikov-Shubin_invariant}
  Another consequence of~\eqref{uniform_estimate} is that the first 
  Novikov-Shubin invariant of $\overline{X}$ with respect to ${\mathcal N}(G)$ is
  greater or equal to $1$. This has already been proved in general, actually,
  one does know the value in terms of $G$, see for
  instance~\cite[Theorem~2.55~(5) on page~98]{Lueck(2002)}.
\end{remark}


\section{Higher dimensions}
\label{sec:higher_dimensions}

\begin{conjecture}[Higher dimensions]\label{con:higher_dimensions}
  For every finite $CW$-complex $X$ and $p \ge 1$, there exists positive
  constants $C$, $\epsilon$ and $\alpha$ such that the $p$-th spectral density
  function $F_p(X_i) = F(c_p(X_i))$ of $X_i$ satisfies for all $\lambda \in
  [0,\epsilon)$ and all $i = 1,2, \ldots$
  \begin{eqnarray*}
    \frac{F_p(X_i)(\lambda) - F_p(X_i)(0)}{[G:G_i]} & \le & C \cdot \lambda^{\alpha}.
  \end{eqnarray*}
\end{conjecture}

This conjecture seems to be hard but is very interesting. It would imply for instance
the conjecture that all Novikov-Shubin invariant of a $G$-covering $\overline{M} \to M$ 
of a closed smooth Riemannian manifold $M$ are positive, see~\cite[Conjecture~7.2]{Lott-Lueck(1995)}
provided that $G$ is residually finite.

Conjecture~\ref{con:higher_dimensions} 
is equivalent to the one described in the following Remark~\ref{rem:higher_dimensions_A}.

\begin{remark}[Matrix formulation]\label{rem:higher_dimensions_A}
  Consider a matrix $A \in   M_{m,n}(\IZ G)$. 
   Let $A[i] \in M_{m,n}(\IZ[G/G_i])$ be its image under the
  canonical projection. Let $r_A \colon l^2(G)^m \to l^2(G)^n$ and $r_{A[i]}
  \colon l^2(G/G_i)^m \to l^2(G/G_i)^n$ be the induced operators given by right multiplication
with $A$ and $A_i$. Then there is
  the conjecture that there exists constants $C,\alpha,\epsilon > 0$ such that
  for all $i = 1,2,\ldots$ and $\lambda \in [0,\epsilon]$
\begin{eqnarray*}
\frac{F(r_{A[i]})(\lambda) - F(r_{A[i]})(0)}{[G:G_i]} & \le & C \cdot \lambda^{\alpha}.
\label{uniform_estimate_matrices}
\end{eqnarray*}
\end{remark}

\begin{remark}[Approximating Fuglede-Kadison determinants]
\label{rem:Approximating-Fuglede-Kadison_determinants}
The conjecture appearing in Remark~\ref{rem:higher_dimensions_A}
implies by the strategy mentioned in the proof of 
Theorem~\ref{the:Determinant_approximation_for_graphs}
\[
\ln\left({\det}^{(2)}\bigl(r_A);{\mathcal N}(G)\bigr)\right) 
= \lim_{i \to \infty} \frac{\ln\bigl({\det}^{(2)}(r_{A[i]})\bigr)}{[G:G_i]}.
\]
The   last   equation   has   been   proved   for   $G   =   \IZ$   by
Schmidt~\cite{Schmidt(1995)},    see    also~\cite[Theorem~15.53    on
page~478]{Lueck(2002)} and hence holds for all virtually cyclic groups
$G$. To the author's knowledge it seems to be unknown for all infinite
groups which are not virtually  cyclic.  For more information we refer
to the discussion in~\cite[Section~3.1]{Li-Thom(2012)}.
\end{remark}

\begin{remark}[Approximating $L^2$-torsion]
\label{rem:approximating_L2-torsion}
If  one can prove the equivalent claims in
Conjecture~\ref{con:higher_dimensions} or
Remark~\ref{rem:higher_dimensions_A}, then 
Conjecture~\ref{con:Approximation_conjecture_for_analytic_L2-torsion} 
is true, provided that 
all $L^2$-Betti number of $\overline{M}$ vanish, see~\cite{Lueck(2012_l2approx)}.
\end{remark}




\end{document}